\documentclass[11pt, letterpaper]{amsart}
\usepackage[utf8]{inputenc}
\usepackage{enumerate, xcolor, graphicx}
\definecolor{Magenta}{cmyk}{0,1,0,0}
\definecolor{dgreen}{rgb}{0,0.6,0.}

\usepackage{academicons}

\usepackage{amssymb,amsthm,amsmath,amsfonts}
\usepackage{diagbox}

\usepackage{mathrsfs, mathtools, mathdots}
\usepackage{paralist,fancyhdr,etoolbox,nicefrac,mathpazo}
\usepackage{listings}
\usepackage{tikz-cd}
\usepackage{parskip,xfrac}
\usepackage{datetime}
\usepackage{hyperref}
\usepackage{cleveref}
\usepackage{etoolbox}
\usepackage[top=27mm,bottom=27mm,left=25mm,right=25mm]{geometry}
\hypersetup{
    colorlinks=true,
    citecolor=red,       
    linkcolor=darkblue,  
    urlcolor=red         
}
\usepackage[T1]{fontenc}
\usepackage[default,scale=1.5]{comicneue} 
\usepackage[italic]{mathastext} 
\DeclareFontSeriesDefault[rm]{md}{bf}
\DeclareFontSeriesDefault[sf]{md}{bf}

\AtBeginDocument{\bfseries\boldmath}

\let\oldnormalfont\normalfont
\renewcommand{\normalfont}{\oldnormalfont\bfseries}
\boldmath
\usepackage[normalem]{ulem}     
\newcommand{\ultrabold}{\fontseries{b}\selectfont}
\definecolor{darkblue}{rgb}{0.0, 0.0, 0.55} 

\makeatletter
\def\section{\@startsection{section}{1}%
  \z@{.7\linespacing\@plus\linespacing}{.5\linespacing}%
  {\normalfont\Large\bfseries\itshape\ultrabold\color{darkblue}}}
  
\def\subsection{\@startsection{subsection}{2}%
  \z@{.5\linespacing\@plus.7\linespacing}{-.5em}%
  {\normalfont\large\bfseries\itshape\ultrabold\color{darkblue}}}
\makeatother

\makeatletter
\let\oldsection\section
\renewcommand{\section}{\@ifstar{\@starredsection}{\@unstarredsection}}
\newcommand{\@starredsection}[1]{\oldsection*{\uline{#1}}}
\newcommand{\@unstarredsection}[1]{\oldsection{\uline{#1}}}

\let\oldsubsection\subsection
\renewcommand{\subsection}{\@ifstar{\@starredsubsection}{\@unstarredsubsection}}
\newcommand{\@starredsubsection}[1]{\oldsubsection*{\uline{#1}}}
\newcommand{\@unstarredsubsection}[1]{\oldsubsection{\uline{#1}}}
\makeatother

\makeatletter
\newtheoremstyle{underlinedthm}
  {6pt}
  {6pt}
  {\itshape}
  {}
  {\bfseries\itshape\ultrabold\color{darkblue}}
  {.}
  {.5em}
  {\uline{\thmname{#1}\thmnumber{ #2}\thmnote{ (#3)}}}

\newtheoremstyle{underlineddef}
  {6pt}
  {6pt}
  {\normalfont}
  {}
  {\bfseries\itshape\ultrabold\color{darkblue}}
  {.}
  {.5em}
  {\uline{\thmname{#1}\thmnumber{ #2}\thmnote{ (#3)}}}
\makeatother

\theoremstyle{underlinedthm} 
\makeatletter
\def\@settitle{\begin{center}%
  \baselineskip14\p@\relax
  \huge\bfseries\itshape\ultrabold\color{darkblue}%
  \expandafter\uline\expandafter{\@title}
  \end{center}%
}
\makeatother

\newtheorem{theorem}{Theorem}[]
\newtheorem{definition}{Definition}[]
\newtheorem{discussion}{Discussion}[]

\newtheorem{lemma}{Lemma}
\newtheorem{proposition}{Proposition}
\newtheorem{notation}[definition]{Notation}
\newtheorem{corollary}{Corollary}
\newtheorem{conjecture}{Conjecture}
\theoremstyle{remark}
\newtheorem{remark}[theorem]{Remark}

\newcommand{\Zdisc}{\mathfrak{disc}_{\mathbb{Z}}}
\newcommand{\disc}{\mathfrak{disc}}
\newcommand{\ncp}[1]{\langle#1\rangle}

\newcommand{\Q}{\mathbb{Q}}
\newcommand{\Qbar}{\overline{\mathbb{Q}}}
\newcommand{\Z}{\mathbb{Z}}

\newcommand{\Zbar}{\overline{\mathbb{Z}}}
\newcommand{\R}{\mathbb{R}}
\newcommand{\C}{\mathbb{C}}
\newcommand{\K}{\mathbb{K}}

\renewcommand{\P}{\mathbb{P}}
\newcommand{\A}{\mathbb{A}}

\newcommand{\p}{\mathfrak{p}}
\newcommand{\q}{\mathfrak{q}}
\renewcommand{\O}{\mathcal{O}}
\newcommand{\OK}{\mathcal{O}_{\mathbb{K}}}
\newcommand{\nm}{\mathfrak{Nm}}
\newcommand{\F}{\mathbb{F}}

\newcommand{\Proj}{\mathbb{Proj}}
\renewcommand{\pounds}{\eth}

\begin{document}

\title[Pseudo and Sudo maximal Rings]{Weighted enumeration of number-fields and counting points that take bounded squarefree values along certain polynomials using Pseudo and Sudo maximal orders}
\author[G.D.Patil]{Gaurav Digambar Patil}
\address{LMSI address}
\email{gaurav.patil@lmsi.org}
\href{https://orcid.org/0009-0008-7073-0160}{orcid.org/0009-0008-7073-0160}
\date{\today}

\begin{abstract}
    Dedekind Kummer for Binary rings (from \cite{Corso2005DecompositionOP}) allows us to relate the local decomposition of a binary form to the local decomposition of the binary ring. This allows us to relate the irreducible decomposition (\cite{patilUniqueDecomp}) of the binary ring to the local decomposition of the binary form. On the other hand, the root distribution (local) and hence irreducible decomposition(local) is connected to the discriminant and the higher singular/non singular loci of the discriminant form. This allows us to relate the locally non singular binary forms (non-singular along the discriminant) with certain local decompositions of the binary form. We define a subclass of irreducible orders called \textit{pseudo-maximal orders}, and subsequently define \textit{sudo-maximal orders} as those whose irreducible decomposition consists entirely of pseudo-maximal orders to mimic this local behaviour of a binary ring not on the locally non singular locus of the discriminant form. Because these rings are ``close'' to being maximal (i.e., the ring of integers) one can enumerate the number of such rings in a number-field very easily, which we do. Also, this local structure (called Sudo-maximal ism) is characterized by being close to maximal it is shown by  any intermediate order as well. This means it is shown by weakly divisible rings (\cite{patil2024weaklydivisiblerings}) coming from similarly restricted binary forms as well. As noted in \cite{patil2024weaklydivisiblerings}, these rings allow us to get number-fields associated to relatively larger height binary form, but with much smaller discriminant.
        
    Noting that the local non-singularness is restrictive enough that all intermediate rings of a form that is non singular are in fact weakly divisible, including the ring of integers and thus are able to realize the discriminant of the ring of integers in this case (which has to be squarefree) as a polynomial in terms of some expanded coefficients parametrizing weakly divisible rings. We thus construct a polynomial $\pounds_n$ whose squarefree locus is in one-to-one correspondence with this locally non-singular locus. 
    
    To estimate the number of locally non-singular binary forms in appropriate regions, we use a version of the Ekedahl sieve tailored to the discriminant polynomial for this type of counting in \cite{patilEkedahlSingularDiscriminant}. These estimates in the above contexts give us, a weighted count for the number of degree-$n$ number-fields with bounded discriminant, as well as a weighted lower bound for number-fields whose rings of integers admit certain geometric presentations like $\Z[\alpha,\beta]$ where $\alpha\beta\in \Z$ and a lower bound for how often does $\pounds_n$ takes bounded squarefree values up to obvious symmetries of $\pounds_n$
\end{abstract}
\maketitle
\section{Introduction}
We define the quantities $S(n,X)$ and $N(n,X)$ as follows:
\begin{definition}
    \begin{align*}
    S(n,X)&:=\{\K:[\K:\Q]=n, Gal(\overline{\K}/\Q)\simeq S_n,\disc(\K)<X\} \\
    N(n,X)&:=\#S(n,X)
    \end{align*}   
\end{definition}
Malle's conjecture (for $S_n$) \cite{MalleConjII} predicts that
    \[
    N(n,X)\simeq c_n X
    \]
where $c_n$ is a positive constant depending only on $n.$ This conjecture is known for small values of $n$, namely $n\le 5.$ 

In this paper, we establish the following asymptotic lower bound:
\begin{theorem}
    We show that $N(n,X)\gg_n X^{\frac{1}{2}+\frac{1}{n-\frac{4}{3}}-\epsilon}$ infinitely often.
\end{theorem}
The history of lower bounds is summarized in \cref{tab:lower_bounds}.

\begin{table}[htbp]
    \centering
    \renewcommand{\arraystretch}{2} 
    \begin{tabular}{|p{2.5cm}|p{3.25cm}|p{5cm}|p{4cm}|}
        \hline
        \textbf{Authors and Source} & \textbf{Lower Bound for $N(n,X)$} & \textbf{Remarks} & \textbf{Geometry} \\
        \hline
        Malle (2002) \cite{MalleConjI} & 
        $X^{\frac{1}{n}}$ & 
        & $\Z[\alpha]\subseteq\OK$ where $\alpha$ satisfies a very specific type of polynomial only. $\A^1(\Zbar).$
        \\
        \hline
        Ellenberg-Venkatesh (2006) \cite{Ellenberg2003TheNO} & 
        $X^{\frac{1}{2}+\frac{1}{n^2}}$ & 
        They actually showed that the quantity is $\gg X^{\frac{1}{2}+\frac{1}{n}}$ infinitely often, though showing this universally ran into technical difficulties. &  
        Examines all possible polynomials bounded by height and organizes roots by the number of $\alpha\in \OK$ with bounded discriminant, coming from the height. $\A^1(\Zbar).$
        \\
        \hline
        Bhargava-Shankar-Wang \cite{BSW-monicsquarefree} & 
        $X^{\frac{1}{2}+\frac{1}{n}}$ &
        & 
        Counts number-fields $\K$ in which one can find $\alpha\in \OK$ with $\Z[\alpha]=\OK$ of limited height, using local conditions to eliminate redundancies from the previous strategy. $\A^1(\Zbar).$
        \\
        \hline
        Bhargava-Shankar-Wang \cite{bhargava2022squarefree} & 
        $\gg_n X^{\frac{1}{2}+\frac{1}{n-1}}$ & 
        & 
        Counts number-fields $\K$ in which one can find $\delta\in \K$ with $\Z[\delta]\cap \Z[\delta^{-1}]=\OK$ of limited height. Shifts parametrization space from monic polynomials to binary forms. $\P^1(\Zbar).$\\
        \hline
        Patil \cite{patil2024weaklydivisiblerings} (Our previous paper) & 
        $\substack{ X^{\frac{1}{2} +\frac{1}{n-1} + \frac{1}{5n^3}+\epsilon}: n \text{ odd} \\ X^{\frac{1}{2} +\frac{1}{n-1} + \frac{1}{88n^8}+\epsilon}: n \text{ even} }$ & 
        
        & Rings associated to points satisfying $aXY-bZ^2$ in $\P^2(\Zbar).$
        \\
        \hline
    \end{tabular}
    \caption{Summary of lower bounds for $N(n,X)$}
    \label{tab:lower_bounds}
\end{table}
Note that in \cite{patil2024weaklydivisiblerings}, we established a lower bound for rings with discriminant up to $X$ of approximately $X^{\frac{1}{2}+\frac{1}{n-4/3}}$. However, the corresponding number-field count fell short of this bound, as the sieve used to distinguish number-fields via local conditions does not apply to the highly skewed space parametrizing weakly divisible rings. In this paper, we develop a sieve—building upon the work of Ekedahl, Poonen, Stoll, and Bhargava—that operates effectively in these skewed regions while enforcing modular conditions within the error terms. This allows us to yield a weighted count of various types of number-fields bounded by discriminant. 

\begin{definition} We define 
    \[
        N^*(n,X):=\sum_{\K\in S(n,X)}\frac{1}{\sqrt{\disc(\K)}}.   
    \]
\end{definition}

\begin{remark}
    Malle's conjecture for $N(n,X)$ is equivalent (via partial summation) to 
    \[
        N^*(n,X)\simeq 2c_n X^{\frac{1}{2}}.
    \]
\end{remark}

\begin{theorem} We show that for $n\ge 5$,
    \[
    N^*(n,X)\gg_n X^{\frac{1}{n-\frac{4}{3}}} \frac{(\log\log X)^{\frac{2n^2-6n+4}{3n^2-13n+12}}}{(\log X)^{\binom{n}{2}-1 +\frac{2n^2-6n+4}{3n^2-13n+12}}}.\]
\end{theorem}

\begin{remark}
    We can interpret $\frac{N^*(n,X)}{2\sqrt{X}+\zeta(1/2)}$ as the weighted average of the number of number-fields of a given discriminant less than $X$. Thus, we may view 
    \[
    \frac{N^*(n,X)}{2\sqrt{X}+\zeta(1/2)}\cdot X
    \]
    as a weighted lower bound for the number of number-fields with bounded discriminant. We will thus show that the weighted lower bound for the number of number-fields with discriminant bounded by $X$ is
    \[
    \gg X^{\frac{1}{2}+\frac{1}{n-\frac{4}{3}}-\epsilon}.
    \]
\end{remark}

As a corollary, our result implies 
\[
\limsup_{X\longrightarrow\infty} \frac{N(n,X)}{X^{\frac{1}{2}+\frac{1}{n-\frac{4}{3}}-\epsilon}}=\infty.
\]
We also deploy this approach to count how often a specific multivariate polynomial $\pounds_n$ (see \cref{defining pounds}) takes a squarefree value bounded by $X$. This polynomial is structurally similar to the discriminant polynomial in terms of its coefficients. More specifically, we show:
\begin{theorem}
If $S(X)$ denotes a set of points $P$ such that $\pounds_n(P)$ is squarefree and $\le X$, such that for any distinct points $Q,Q'\in S(X)$, $Q$ is not equivalent to $Q'$ up to the action given in \cref{symmetriesofpounds}, then 
\[
\#S(X)\gg_n\frac{X^{\frac{1}{2}+\frac{1}{n-1.5}}(\log \log X)^{\frac{1}{n-1.5}}}{(\log X)^{1+\frac{1}{n-1.5}}}
\]
provided that $n\ge 5.$
\end{theorem}
We achieve this by demonstrating a bijection between elements satisfying the Ekedahl sieve conditions for the discriminant polynomial and points at which $\pounds_n$ takes squarefree values.

Given the interest in enumerating distinct number-fields whose rings of integers possess a specific structure (i.e., enumeration of number-fields parameterized by certain varieties), we also deploy this sieve on monic polynomials rather than binary forms.
\begin{theorem}
    If $a_s$ denotes the number of number-fields with squarefree discriminant $s$ whose ring of integers can be written as $\Z[\alpha,\beta]$ where $\alpha\beta\in \Z,$ then 
    \[
    \sum_{s\le X} \frac{a_s}{\sqrt{s}}\gg_n X^{\frac{1}{n-1}+\frac{1}{n(n-1)(n-2)}}\frac{(\log\log X)^{\frac{1}{n-1}+\frac{1}{n(n-1)(n-2)}}}{(\log X)^{\binom{n}{2}-1+\frac{1}{n-1}+\frac{1}{n(n-1)(n-2)}}}.
    \]
\end{theorem}

Since all of our sieves capture number-fields with squarefree discriminants, we can also apply these results to unramified $A_n$-extensions of $\Q(\sqrt{s})$ as $s$ varies.
\begin{theorem} If $\frac{\log|Cl_{\Q(\sqrt{s})}[A_n]|}{\log|n!/2|}$ denotes the number of unramified $A_n$-extensions of $\Q(\sqrt{s})$, then
    \[
    \sum_{s<X}\frac{\log|Cl_{\Q(\sqrt{s})}[A_n]|}{\sqrt
    {s}}\gg_n X^{\frac{1}{n-\frac{4}{3}}}\frac{(\log\log X)^{\frac{2n^2-6n
    +4}{3n^2-13n+12}}}{(\log X)^{\binom{n}{2}-1 +\frac{2n^2-6n
    +4}{3n^2-13n+12}}}
    \]
\end{theorem}
For comparison, Theorem 4 in \cite{bhargava2022squarefree} shows that $\sum_{s<X}\log|Cl_{\Q(\sqrt{s})}[A_n]|\gg_n X^{\frac{1}{2}+\frac{1}{n-1}}$.
\section{Preliminaries: Weakly Divisible Forms and Associated Rings}

Before we recall the definition of weakly divisible ring, let us briefly recall the definition of Binary Basis  associated to Binary forms (first defined (as far as we can trace) in Birch Merriman finiteness theorems (\cite{BirchMerrimanFiniteness}) and later used by various papers (\cite{bhargava2022squarefree},\cite{MM-ONEFINE}\cite{Corso2005DecompositionOP}\cite{Nakagawa1989}) which builds into Weakly Divisible rings (\cite{patil2024weaklydivisiblerings}). We give this for completeness sake.
\par
\begin{definition}
If
\begin{equation*}
  f(X,Y):= a_0X^n+a_{1}X^{n-1}Y +\cdots +a_nY^n
\end{equation*} 
denotes a binary form of degree $n$ with $a_0\neq 0$ and  if $\delta$ denotes the image of $X$ in the algebra $\Q[X]/(f(X,1)),$ then we define

\begin{equation*}
    \ncp{B_0,B_1,\cdots,B_{n-1}}:= \ncp{1,\: a_0 \delta , \: a_0\delta^2+a_{1}\delta,\:\cdots, \:\sum_{i=0}^{k-1}a_{i}\delta^{k-i},\:\cdots,\:\sum_{i=0}^{n-2}a_{i}\delta^{n-1-i}}
\end{equation*}
and define $R_f$ as the $\Z$ span of the basis above, that is
\begin{equation}\label{basis-binary}
    R_f:=\Z\ncp{B_0,B_1,\cdots, B_k,\dots,B_{n-1}}.
\end{equation}

\end{definition}
\begin{remark}
    We refer to this basis as the canonical basis attached to $f.$  Also, there is an annoying way to  define this basis to make it independent of the condition $a_0\neq 0$ by defining it using the multiplication table...but why bother, this is far more illuminating for our purposes and we avoid reducible polynomials anyways. 
\end{remark}

The following definition is a slight generalization of the one in \cite{patil2024weaklydivisiblerings}.
It is far more symmetric and amenable to $GL_2$ action, however the original is sufficient for counting purposes as we only apply it to Ultra Weakly Divisible polynomials.

\begin{definition}\label{Weakly-Divisible polynomials}
We say a binary form $f(x,y)\in \Z[x,y]$ of degree $n$ is weakly divisible by $m\in \Z_{>0}$ if there exists an $\ell \in \P(\sfrac{\Z}{m\Z})$ such that for any lift $[a:b]\in \P(\Z)$ (with $\gcd(a,b)=1$) of $\ell,$ we have 
\begin{align*}
    & m^2|f(a,b)\\
    &m|\frac{\partial f}{\partial x}(a,b) \textit{ and }  m|\frac{\partial f}{\partial y}(a,b)
\end{align*}
Thus, if $\ell\in \sfrac{\Z}{m\Z}$, then $f$ is weakly divisible by $m$ at $\ell$ if and only if $m^2|f(\ell,1)$ and $m|\frac{\partial f}{\partial x}(\ell,1).$
\end{definition}
\begin{definition}\label{notation $f_l$}
    Given a binary form $f$ we set $f_l=f_l(x,y):=f(x+ly,y).$
\end{definition}
\begin{discussion}\label{P(Z/mZ)representation}
    We may represent elements of $\P(\sfrac{\Z}{m\Z})$ as $(m_1,m_2,l)$ where $l\in \Z/m_2\Z$, $m_1m_2=m$ and $\gcd(m_1,m_2)=1.$ In traditional terms, if $[a:b]\in \P(\sfrac{\Z}{m\Z})$ then we may set $m_1=\gcd(m,a)$, $m_2=\frac{m}{m_1}$ and $l$ is defined by $al\equiv b\bmod m_2.$ In this set up $f$ is weakly divisible at $(m_1,m_2,l)$ if and only if $f_l$ has the following form
\begin{equation*}
    f_l=m_1^2a_0x^n+m_1a_{1}x^{n-1}y+a_{2}x^{n-2}y^2+\cdots+a_{n-2}x^2y^{n-2} +m_2a_{n-1}xy^{n-1}+m_2^2a_ny^n.
\end{equation*}
\end{discussion}
See \cite{patil2024weaklydivisiblerings} for context for this next definition.
\begin{definition}
    If $\ncp{B_0,B_1,\cdots B_{n-1}}$ denotes the basis associated to $f$ in \cite{Nakagawa1989}, then we define
    \begin{equation*}R'_{(f,m_1,m_2)}:=\Z\ncp{B_0,\frac{B_1}{m_1},\cdots,B_{n-2},\frac{B_{n-1}}{m_2}}.
\end{equation*}
\end{definition}

We recall the following theorem 4 from \cite{patil2024weaklydivisiblerings}.
\begin{theorem}\label{WeaklyDivisibleDefiningTheorem}
If $f$ is weakly divisible by $m$ at $l\in \frac{\Z}{m\Z},$ then $R'_{(f_l,1,m)}$ is a ring.
We call this ring as the weakly divisible ring associated to $(f,l,m)$ and denote it by $R'_{(f,l,m)}$.
\end{theorem}
\begin{remark}\label{discriminint of R'}
    Since, discriminant of $\ncp{B_0,B_1,\cdots B_{n-1}}=\disc(f),$ we get $\disc(R'_{(f,l,m)})=\frac{\disc(f)}{m^2}$
\end{remark}
\begin{corollary}\label{monicUWD}
    If $f$ is monic and weakly divisible by $m$ at $l$ then $R_{f_l,m}\simeq \Z[\alpha,\beta]$ where $\alpha \beta =\frac{f(l)}{m}.$
\end{corollary}
Clearly, one can  modify the above theorem in the context of the more symmetric definition of weakly divisible.
\begin{theorem}\label{WeaklyDivisibleDefiningTheorem2}
If $f$ is weakly divisible at $(m_1,m_2,l)$ then $R'_{(f_l,m_1,m_2)}$ is a ring.
We call this ring as the weakly divisible ring associated to $(f,m_1,l,m_2)$.
We will often refer to $R'_{(f_l,m_1,m_2)}$ by $R'_{(f,m,l*)}$ where $l^*\in \P(\sfrac{\Z}{m\Z})$ corresponds to $(m_1,m_2,l)$ in \cref{P(Z/mZ)representation}.
\end{theorem}
Note that the above theorem does not need $\gcd(m_1,m_2)=1.$

\section{Preliminaries: Fundamental Theorem of Orders and irreducible orders.}
Recall the following definitions and theorems from \cite{patilUniqueDecomp}.
\begin{notation}
    Given an integral domain $R,$ we define 
        \[
        M(R):=\{\p\subseteq R: \p \textit{ is a non-zero prime ideal in } R\}.
        \]
\end{notation}
\begin{notation}
    We set 
    \[
    S(R):=\{\rho\in M(R): R_\rho \textit{ is not a DVR.}\}.
    \]
\end{notation}
\begin{theorem}[Fundamental Theorem of Orders]\label{FTO}
    Every order $\O$ can be written as an intersection of irreducible orders in a unique way such that the conductors of the irreducible orders are pairwise co-prime.
    The index (and conductor) of $\O$ in $\OK$ will be the product of the indices (and conductors) of the irreducible orders in $\OK$ in the given decomposition. This representation is given by 
    \[
    \O=\cap_{\rho\in S(\O)}(\O_\rho\cap \OK)
    \]
    where the localization of the order for a prime in $S(\O)$ matches with the localization of the corresponding irreducible component, namely $(\O_\rho\cap \OK)$.
\end{theorem}
\begin{notation}
    For an integral domain $R$, we denote the integral closure of $R$ by $\overline{R}$.
\end{notation}
\begin{definition}
     For $\rho\in M(R),$ we denote the set elements $v$ in $M(\overline{R})$ satisfying $v\cap R=\rho$ by $\overline{\rho}$.
\end{definition}
\begin{definition}
    For any integral domain $R,$ whose fractional-field is the number-field $\K,$ and any non-zero prime ideal of $\rho\in M(R),$ we define 
    \[
    ef(\rho):=\sum_{v\in \overline{\rho}}e_v \cdot f_v.
    \] 
    Let $\rho\cap \Z=(p)$.
    We then define $f(\rho)$ by $|\O/\rho|=p^{f(\rho)}$.
    We define $e(\rho)=\frac{ef(\rho)}{f(\rho)}.$
\end{definition}
\begin{definition}\label{ef-definition}
    If $\O$ is an irreducible order with $S(\O)=\{\rho\},$ we define 
    \[
    ef(\O):=ef(\rho), f(\O)=f(\rho) \textit{ and } e(\O):=e(\rho).
    \]
\end{definition}
\begin{lemma}\label{Normefrelation}
    If $p$ is a prime over $\Z$, has primary decomposition $\prod_i q_i=\cap_i q_i$ in some integral domain $R$, whose fractional field is the number-field $\K,$ where $\q_i$ is $\rho_i$-primary ideal then $\nm(q_i)=p^{ef(\rho_i)}.$ In other words, $ef(\rho)$ also denotes the $\Z_p$-rank of $\varprojlim \sfrac{R}{\rho^k}$.
\end{lemma}
\begin{proposition}[Bijection:Irreducible orders-Local algebras in p-adic extensions]
    There is a natural bijection between local sub-rings of $\prod_{v\in S}\O_{\K,v}$ and irreducible orders $\O\subsetneq \OK$ satisfying $\sqrt{c_{\O}}=\prod_{v\in S}v$ given by 
    \[ 
    \O\longrightarrow \varprojlim \sfrac{\O}{\rho^k}.
    \]
    The inverse of this map is given by intersecting the given local ring in $\prod_{v\in S}\O_{\K,v}$ with $\OK$ as it sits diagonally in $\prod_{v\in S}\O_{\K,v}$. We thus get the following proposition.
\end{proposition}
\section{Classification of Irreducible orders $\O$ with $ef(\O)=1$ and $ef(\O)=2$}
\subsection{$ef(\O)=1$}\hfill\vspace{5pt}

There exist no irreducible order $\O$ in $\O_K$ with $ef(\O)=1$ as this would be akin to finding proper unital sub-ring of the p-adic integers $\Z_p$.
\subsection{$ef(\O)=2$}
\begin{lemma}\label{Pseudomax-precursor}
    If $A=\Z_p+\Z_p\cdot t$ is an algebra of rank $2$ over $\Z_p$ then all sub-rings of $A$ which have rank $2$ over $\Z_p$ and containing $\Z_p$ are given by $\Z_p +p^kA$ for some $k\in \Z_{\ge0}.$ The index of $\Z_p +(p^k)$ in $A$ is $p^k.$
\end{lemma}
\begin{proof}
    If $M$ is a sub-ring of $A$ which is rank $2$ over $\Z_p$ may be seen as $\Z_p +\Z_pr$ where $r=a+bt$ with $a,b\in \Z_p$.
We further write $b=p^ku$ where $u\in \Z_p$ is a unit, i.e. $u\Z_p=\Z_p.$ Clearly,
    $$
    \Z_p +\Z_p r = \Z_p +\Z_p bt=(\Z_p+\Z_p u p^k)+\Z_p u p^k t = \Z_p + (\Z_p u p^k+\Z_p u p^k t) =\Z_p +p^kA.
$$
\end{proof}
We give a structure theorem for irreducible orders with $ef(\O)=2$.
Similar Structure theorems can be given easily using corresponding ring classification theories for irreducible orders with $ef(\O)\le 4.$ With some difficulty, the authors also expect such theorems possible for $ef(\O)=5$ by extending Bhargava's work for quintic rings to $\Z_p.$
\begin{definition}
     A {\em Pseudo Maximal Order} in $\K$ is an irreducible order $\O$ satisfying $ef(\O)=2.$ See \cref{ef-definition} for the definition of $ef.$
\end{definition}
If $\O$ is an irreducible order with $S(\O)=\{\rho\}$, then its conductor is made up of prime ideals (in $\K$) in $\overline{\rho}$.
Since there is a limited number of ways of appropriately partitioning $2$, we get the following.
\begin{lemma}\label{possiblelocalPseudo}
    If $\O$ is a Pseudo Maximal Order, then one of the following must be true: 
    \begin{enumerate}[A)]
        \item  $\overline{\rho} = \{v,w\} \textit{ with } f_v=f_w=e_v=e_w=1.$
        \item $\overline{\rho} = \{v\} \textit{ with } f_v=2=2e_v.$
        \item $\overline{\rho} = \{w\} \textit{ with } 2f_v=2=e_v.$
    \end{enumerate}
\end{lemma}
\begin{theorem}[Pseudo Maximal Classification Theorem A]\hfill\\
    If $\{v,w\}\subseteq M(\K) \textit{ with } f_v=f_w=e_v=e_w=1,$ then the irreducible orders with conductor $v^a w^b$ are given by 
 
   \begin{itemize}
        \item if $a=b=r$ then there is a unique irreducible order $\O$ with conductor $(vw)^r$ given by $\Z+(vw)^r.$ This order will have index in $\OK$ given by $[\OK:\Z+(vw)^r]=p^r.$
        \item if $a\neq b$ then there is no irreducible (or otherwise) order with conductor $v^aw^b.$
    \end{itemize}
\end{theorem}
\begin{proof}
    Clearly, we are looking for irreducible orders in $\Z_p\oplus\Z_p$ which is treated an extension of $\Z_p$ via diagonal embedding.
So, we may write it as $\Z_p+\Z_p t$ where $t=(0,1)$.
Thus every order here must be of the form $\Z_p+ p^k(\Z_p\oplus\Z_p).$ Coordinate wise $p=(p,p)$ due to diagonal embedding.
Thus, the sub-ring is given by $\Z_p + ((p,1)^k \cdot (1,p)^k)(\Z_p\oplus\Z_p)$.
Intersecting with $\OK$ we get that the corresponding order is given by $\Z+(vw)^k$ whose conductor is easily seen as $(vw)^k.$
\end{proof}
\begin{theorem}[Pseudo Maximal Classification Theorem B]\hfill\\
    If $\{v\}\subseteq M(\K) \textit{ with } f_v=2=2e_v,$ then there is a unique irreducible order with conductor $v^a$ which is given by $\Z+v^a$ with $[\OK:(\Z+v^a)]=p^a$ 
\end{theorem}
\begin{proof}
    Clearly, we are looking for irreducible orders in the ring of integers of the unique totally un-ramified extension of degree 2 over $\Z_p.$ Thus, $p$ may be seen as the uniformizer for this completion.
So, we may write it as $O_v:=\Z_p+\Z_p t$. Thus every order here must be of the form $\Z_p+ p^k(\O_v).$ Thus, the sub-ring is given by $\Z_p + p^k(O_v)$.
Intersecting with $\OK$ we get the order is given by $\Z+v^k$ whose conductor is easily seen as $v^k.$
\end{proof}
\begin{theorem}[Pseudo Maximal Classification Theorem C]\hfill\\
    If $\{v\}\subseteq M(\K) \textit{ with } 2f_v=2=e_v,$ then 
    \begin{itemize}
        \item if $a\ge 2$ and $a$ is even then there is a unique irreducible order with conductor $v^a$ which is given by $\Z+v^a$ with $[\OK:(\Z+v^a)]=p^{a/2}.$ 
        \item if $a\ge 1$ and $a$ is odd then there is no irreducible (or otherwise) order with conductor $v^a.$
    
\end{itemize}
\end{theorem}
\begin{proof}
    We are looking for irreducible orders in the ring of integers of some totally ramified extension of degree 2 over $\Z_p.$ Thus, $p$ may be seen as the square of the uniformizer for this completion.
So, we may write it as $O_v:=\Z_p+\Z_p t$ where $t$ is the uniformizer (since it is totally ramified. Thus, every order here must be of the form $\Z_p+ p^k(\O_v).$ Thus, the sub-ring is given by $\Z_p + p^k(O_v)$. Intersecting with $\OK$ we get the order is given by $\Z+v^{2k}$ (as $(p)=v^2$, as it is a totally ramified extension) whose conductor is easily seen as $v^{2k}.$
\end{proof}

\begin{corollary}\label{uniquePseudoMax}
    Given a set of valuations $S\subseteq M(\K)$ such that $\sum_{v\in S}e_vf_v=2$ then given $r\ge 1$ there exists a unique irreducible order $\O$ with $S(\O)=\{\rho\}$ such that $\overline{\rho}=S$ and $[\OK:\O]=p^r.$ 
\end{corollary}
\begin{definition}\label{Def:Sudo and ReSMO}
 
   A {\em Sudo Maximal Order} in $\K$ is an order $\O$ all of whose irreducible order components are Pseudo-maximal. In other words, $\O$ is Sudo Maximal if 
    \begin{align*}
        \forall \rho \in S(\O): ef(\rho)=2.
    \end{align*}
    
    A Restricted Sudo Maximal Order (or ReSMO) is a Sudo Maximal Order satisfying : for every prime $p\in \Z$ we have at most one prime ideal in $S(\O)$ which contains $p.$   
\end{definition}
\begin{remark}\label{truSudo}
The author wanted to define Pseudo Maximal to mean $ef(\O)/f(\O)\le 3$ and define Sudo maximal from there. This is more natural as we will show in an upcoming paper soon. Maybe the author should water down the next paper and only use this definition. Alas, we will instead focus on counting square-frees. 
\end{remark}
\begin{discussion}
    Note that, \cref{possiblelocalPseudo} and \cref{uniquePseudoMax}, the number of pseudo maximal orders with index $p^k$ in its ring of integers ($\OK$) is equal to the sum of three quantities, namely:
    \begin{enumerate}
        \item The number of primes $v\in M(\K)$ satisfying $e_v=2$, $f_v=1$ and $v\cap \Z=(p).$
        \item The number of pairs of distinct primes $v,w\in M(\K)$ satisfying $1=e_v=f_v=e_w=f_w$, $w\cap \Z=v\cap \Z=(p).$
        \item The number of primes $v\in M(\K)$ satisfying $e_v=1$, $f_v=2$ and $v\cap \Z=(p).$
  
  \end{enumerate}
    As a result, the following lemma follows.
\end{discussion}
\begin{lemma}
    Number of ReSMOs in a number-field $\K$ with index $m$ is bounded above by $\binom{n}{2}^{\omega(m)}$ where $n=[\K:\Q].$
\end{lemma}
\begin{remark}\label{actually1}
    In fact simple chebatorev computations (sizes of class) show that this number on an average will be 1. 
\end{remark}

\begin{theorem}\label{Number of ReSMOs}
    It follows that the number of ReSMOs in a given number field $\K$ with $[\K:\Q]=n$ with discriminant bounded by $X$ is  
    \[
    \ll_n \sqrt{\frac{X}{|\Zdisc(\K)|}}  \log(X)^{\binom{n}{2}-1}
    \]
\end{theorem}
\begin{proof}
    A ReSMO in a number field $\K$ has index $m$ if and only if it 
    has discriminant $m^2\Zdisc(\K)$. Applying the preceding lemma, we can conclude that the number of ReSMOs in a number-      field $\K$ is 
    \[
    \le \sum_{m\le \sqrt{\frac{X}{|\Zdisc(\K)|}}}\binom{n}{2}^{\omega(m)}\ll_n\sqrt{\frac{X}{|\Zdisc(\K)|}}  (\log X-\log|\Zdisc(\K)|)^{\binom{n}{2}-1}.
    \]
\end{proof}
\begin{remark}
    From \cref{actually1}, we see the real asymptotic will actually be $c_\K X$ for some constant dependent on the field. Furthermore, this can be applies to Sudo maximal and even the definition desired in \cref{truSudo}, which includes all orders with index that has powers of primes less than or equal to $3$, and can easily be extended a little bit further even.
\end{remark}
\section{Weakly Divisible rings from Ultra Weakly Divisible Forms yield ReSMOs}
This section explores ``Ultra Weakly divisible rings", which are any intermediate rings between a binary ring associated to binary forms which are not in the singular locus of the discriminant (as a polynomial in the coefficients of the form) for any prime $p\in \Z,$ and establishes our previous insight that such rings are in fact ReSMOs. We also show that all intermediate rings, including the ring of integers in such a case, are weakly divisible rings (see \cref{WeaklyDivisibleDefiningTheorem2}). This allows all such rings of integers to be represented by a well-behaved parameter space (see \cite{patil2024weaklydivisiblerings}).

\begin{definition}
    Recall that when $p\neq 2,$ and $f$ is a binary $n$-ic integral form Bhargava-Shankar-Wang define $f$ is strongly divisible by $p$ (in \cite{BSW-monicsquarefree}) if $p^2|\disc(f+p\cdot g)$ for any choice of binary $n$-ic form $g$.
\end{definition}

\begin{remark}
    In \cite{bhargava2022squarefree}, the authors define ``Weakly divisible" as not strongly divisible. We do not do that. Implicitly they use (show) that if $f$ is not strongly divisible by $p$ and $p^2|\disc(f)$ then $f$ has a double root in $\P^1(\F_p)$ and if this root is at $[a:b]\bmod p$ for some co-prime $a,b\in \Z$, then $p^2|f(a,b).$ The following theorem is implicit in making the above claim. We make this explicit, akin to the monic polynomial version of this theorem used in \cite{BSW-monicsquarefree}.
\end{remark}
\begin{theorem}
    A binary form $f$ is strongly divisible by $p$ if one of the 
    following holds:
    \begin{enumerate}
            \item $f\bmod p$ has a triple root in $\P(\F_p)$
                   
            \item $f\bmod p$ has two double roots in $\P(\overline{\F}_p).$
        \end{enumerate}
\end{theorem}
\begin{proof}
     We will assume $f$ is irreducible and $\K_f:=\Q[x]/(f),$ we will get rid of this assumption at the end of this proof.
    
    If $f(x,y)$ shows the following decomposition/factorization modulo $p,$    
    \[
    f(x,y)\equiv \prod_i (f_i)^{e_i} \bmod p,
    \]
    then we know from Theorem 1 in \cite{Corso2005DecompositionOP} that $pR_f$ must factorize as 
    \[
    \prod_{i=1}^k \q_i= pR_f
    \]
    where $q_i$ are primary ideals corresponding to distinct primes $\p_i$ in $R_f.$

    Now Corollary 1 in \cite{Corso2005DecompositionOP} also tells us that $\p_i$ is invertible if and only if $\p_i^{e_i}=\q_i.$
   
 \begin{itemize}
        \item Case 1: All $\p_i$ are invertible.
In this case, note that none of $\p_i$ contain the conductor of $R_f$, $c_{R_f}.$ Thus, $c_{R_f}$ is co-prime to all of $p_i$ which further implies $p$ is co-prime to $c_{R_f}$ and thus co-prime to $[\O_{\K_f}:R_f]$ since it divides $\vert\OK/(\Z+c_{R_f})\vert$ which in turn divides  $\vert\OK/c_{R_f}\vert=\nm_{\OK}(c_{R_f})$.
In this case we can use Theorem 3 from \cite{Corso2005DecompositionOP} to claim that the ramification index of $p_i\O_{\K_f}$ is $e_i$ and its inertial degree is $deg(f_i).$ Then, by standard algebraic number theory $p^{\sum_{i=1}^{k}deg(f_i)(e_i-1)}$ must divide the discriminant of $\O_{\K_f}$ and thus $R_f$ and thus $\disc(f).$ Thus if some $e_i\ge 3$ or $e_i=2$ and $f_i\ge 2$ then $p^2|\disc(f).$
        \item Case 2: Some $\p$ over $pR_f$ is not invertible.
Thus, $\p$ contains the conductor of $R_f$ or  $\p\in S(\R_f)$ which means that $p$ divides the index of the localization in its integral closure which is same as index of the irreducible component corresponding to $\p$ in $\O_{\K_f}.$ Since, $p|[\O_{\K_f}:R_f]$, we get  $p^2|\Zdisc(R_f)=\disc(f).$
    \end{itemize}  
    
    Thus, if $f$ is irreducible primitive binary $n$-ic form, then $p^2|\disc(f+pg)$ for any choice of $g$ provided that for some $i,$ $deg(f_i)\ge 2 \textit{ and } e_i\ge 2$ and $f+pg$ is irreducible.
Similarly, if $f$ is irreducible primitive binary $n$-ic form, then $p^2|\disc(f+pg)$ for any choice of $g$ provided that for some $i,$ $deg(f_i)\ge 1 \And e_i\ge 3$ and $f+pg$ is irreducible.
To get rid of the irreducible condition we will use Hilbert irreducibility theorem.
Note that $\disc(f+p^2 g)\equiv \disc(f) \bmod p^2.$ Furthermore, note that the number of reducible forms is less $o(1)$ (Hilbert irreducibility) of all forms and thus, one can always perturb a binary form by some form of the type $p^2g$ to make it irreducible.
If this were not so, then a proportion of all polynomials would in fact be $\gg\frac{1}{p^{2n}}$, as all polynomials congruent to our base polynomial modulo $p^2$ would be reducible.
\end{proof}
    \begin{lemma}\label{nodividing a_n-2 and Delta_1}
        If $f(x,y)=a_0x^n+a_1x^{n-1}y+\cdots + a_{n-1}xy^{n-1} +a_ny^n$, $p\neq 2$, $p^2|\disc(f)$, $p|a_n$ $p|a_{n-1}$ and $f \bmod p$ has neither a triple root nor two double roots, then $p\nmid a_{n-3}$ and $p\nmid \disc((f-a_ny^2-a_{n-1}xy^{n-1})/x^2).$
    \end{lemma}
    \begin{proof}
    Referring to Lemma 3 (Disc structure final two) in \cite{patilEkedahlSingularDiscriminant}, we get 
    \[
    \disc(f) \equiv 4 a_n\cdot a_{n-2}^{3}\cdot \disc(\frac{f-a_{n-1}xy^{n-1} -a_ny^n}{x^2}) \bmod p^2.
    \]
    Note if $p|a_{n-2}$ then $f\bmod p$ has a triple root at $0.$ Thus, $p\nmid a_{n-2}$.
Similarly if $p|\disc(\frac{f-a_{n-1}xy^{n-1} -a_ny^n}{x^2})$, then $\frac{f-a_{n-1}xy^{n-1} -a_ny^n}{x^2}\bmod p$ has a double root, and thus $f\bmod p$ must have at least two double roots, namely $0$ and then the double root of $\frac{f-a_{n-1}xy^{n-1} -a_ny^n}{x^2}\bmod p$.
Thus, $p\nmid\disc(\frac{f-a_{n-1}xy^{n-1} -a_ny^n}{x^2}).$
    \end{proof} 
\begin{theorem}\label{Fundamental theorem of rings and polynomials}
    If $f$ is primitive, $p\neq 2$ and $p^2|\disc(f)$ and $f \bmod p$ has neither a triple root nor two double roots then it has a single double root at some point in  $\P(\F_p)$ such that if the double root is at $[a:b]$ where $a,b\in \Z$  and $\gcd(a,b)=1$ then $p^2|f(a,b)$ (and thus, $f$ is not strongly divisible by $p$).
\end{theorem}
\begin{proof}
    If $p^2|\disc(f),$ then $p|\disc(f).$ Since, $\disc(f\bmod p)=(\disc(f)\bmod p)$, $f$ must have double root in $\P(\overline{\F}_p).$ Since, $f$ has neither a triple root nor two double roots in $\P(\overline{\F}_p),$ then $f$ must have a double root in $\P(\F_p).$ If the root is at $[a:b]\in \P(\Z)$ with $\gcd(a,b)=1$ then we can move the root to $0$ by some transformation in $SL_2(\Z)$.
Thus, if 
    \[
    f(x,y)=a_0x^n+a_1x^{n-1}y+\cdots + a_{n-1}xy^{n-1} +a_ny^n
    \]
    then $p|a_{n-1}$ and $p|a_n.$ 
    Referring to Lemma 3 (Disc structure final two) in \cite{patilEkedahlSingularDiscriminant}, we get 
    \[
    \disc(f) \equiv 4 a_n\cdot a_{n-2}^{3}\cdot \disc(\frac{f-a_{n-1}xy^{n-1} -a_ny^n}{x^2}) \bmod p^2.
\]
    \cref{nodividing a_n-2 and Delta_1} gives $p\nmid a_{n-2}$ and $p\nmid \Delta_1.$
    
    It follows that $p^2|\disc(f)\implies p^2|a_n$.
\end{proof}

\begin{definition}\label{UltraWeaklyDivisible}
    We say a  $f$ ultra weakly divisible if $2\nmid\disc(f)$ and for all odd primes $p$
    \[
    p^2|\disc(f) \Rightarrow f \textit{ is not strongly divisible by } p.
\] 
\end{definition}
\begin{theorem}\label{LRGST M}
    If $f$ is ultra-weakly divisible, then there exist a maximal $m_f$ such that $f$ is weakly divisible by $m_f$ (at some $l_f\in \P(\sfrac{\Z}{m\Z})$) and $R'_{(f,m_f,l_f)}$ is maximal i.e. $R'_{(f,m_f,l_f)}$ is THE ring of integers for $\K_f\simeq \Q[x]/(f).$ In fact, if $\disc(f)=st^2$ where $s$ is squarefree, then $m_f=t.$
\end{theorem}
\begin{proof}
Let $p$ denote a prime such that $p^{2}|\disc(f).$ Thus, $\disc(f)\bmod p=0$ and thus, $f$ has a linear double root.
\cref{Fundamental theorem of rings and polynomials} tells us that there is exactly one root of multiplicity exactly two and all other roots are of multiplicity $1$.
Furthermore, WLOG we may move this root to $0\bmod p$ by transforming the polynomial with some element in $SL_2(\Z),$ and thus \cref{Fundamental theorem of rings and polynomials} again tells us that the polynomial may be treated as 
\[
    f(x,y)=a_0x^n+a_1x^{n-1}y+\cdots + a_{n-1}xy^{n-1} +a_ny^n
\]
where $p|a_{n-1}$ and $p^2|a_n$.
Using Lemma 3 (Disc structure final two) in \cite{patilEkedahlSingularDiscriminant}, we see that,
\[
\disc(f)=a_n\cdot \Delta_1 +a_{n-1}^2 \cdot \Delta_2 +a_na_{n-1}\cdot \Delta_3 +a_n^2\Delta_4
\] 
where $\Delta_1 = 4a_{n-2}^3 \disc(\frac{f-a_{n-1}xy^{n-1} -a_ny^n}{x^2}).$ From \cref{nodividing a_n-2 and Delta_1} we also know that $p\nmid \Delta_1.$

if we let $F(x)=f(x,1)$ and $g(x)=f'(x),$ then $g(0)=a_{n-1}\equiv 0\bmod p$ and $p\nmid 2a_{n-2}=g'(0).$ Thus, Hensel's lemma tells us that for each $k$, one can find a lift of $0$ to some $l_k$ $\sfrac{\Z}{p^k\Z}$ such that $p^k|g(l_k)$.
Fix $k$ to be very large, say $v_p(\disc(f))+1$. If 
\[
\begin{bmatrix}1&l_k\\0&1\end{bmatrix}\circ f=A_0x^n+A_1x^{n-1}y+\cdots + A_{n-1}xy^{n-1} +A_ny^n
\]
then $p^k|A_{n-1},$ since $A_n\equiv a_n\bmod p$ (as $l_k\equiv 0\bmod p),$ we get $p|A_n.$ Also $\Delta_1(A_0,A_1,\cdots, A_{n-2})\equiv \Delta_1 \bmod p$ for the same reason and hence $p\nmid \Delta_1(A_0,A_1,\cdots, A_{n-2}).$ Since the discriminant is invariant under the above transformation(and Lemma 3 (Disc structure final two) in \cite{patilEkedahlSingularDiscriminant}, we get 
\[
\disc(f)\equiv A_n(\Delta_1(A_0,A_1,\cdots, A_{n-2}) +A_{n}C_k)\bmod p^{k}
\]
Since, $p\nmid (\Delta_1(A_0,A_1,\cdots, A_{n-2} +A_{n}C_k)$, it follows that $p^{k-1}|A_n$. More specifically $f$ is weakly divisible by $p^{[\frac{D}{2}]}$ at $0.$ Thus, by Chinese Remainder Theorem, we get $f$ is weakly divisible by $m_f$ (where $\disc(f)=sm_f^2$, and $s$ is squarefree) at some $l_f\in \P(\sfrac{\Z}{m_f\Z})$ where 
$\disc(f) =s (m_f)^2$ and $s$ is squarefree. This tells us that $\Zdisc(R'_{(f,m_f,l_f)})=s$. This makes $R'_{(f,m_f,l_f)}$ the ring of integers of $\K_f$ as its discriminant is squarefree (See \cref{WeaklyDivisibleDefiningTheorem2}).
\end{proof}

\begin{theorem}\label{UWDisReSMO}
    The binary ring associated to a ultra weakly divisible irreducible binary form of degree $n$ is a ReSMO, and so are any intermediate rings that contain the binary ring and are contained in the ring of integers.
\end{theorem}
\begin{proof}
    Let $F$ denote an irreducible binary form, such that $F=\prod F_{i}^{e_i}\bmod p$, where $F_i$ are irreducible modulo $p$ then Theorem 1 and Proposition 5 from \cite{Corso2005DecompositionOP} tell us that the primary decomposition of $pR_F$ given is given by $pR_F=\cap \q_i=\prod\q_i$ for the same indexing set, such that $\nm(\q_i)=p^{e_i\deg(F_i)}.$ Let $\rho_i$ denote the radical of $\q_i.$ Let $\K$ denote a number-field generated by $F.$

    Corollary 1 in \cite{Corso2005DecompositionOP} tells us that if $e_i=1$ then $\rho_i\notin S(R_F).$ 

    Note that when $F$ is Ultra Weakly Divisible, we know that $e_i>1 \implies e_i=2$ 
    and $f_i=1.$ Thus, if $e_i=2$ and $f_i=1$ for some prime ideal $\rho_i$.
    Thus, for all $\rho\in S(R_F)$ we have $ef(\rho)=2$ and $f(\rho)=1.$ Thus, $R_F$ is Sudo maximal.
    Furthermore, Ultra Weakly Divisible also implies at most one factor with $e_i>1$ for any given prime $p\in M(\Z),$ making $R_F$ a ReSMO.
    Furthermore, if $R_F\subseteq H\subseteq \OK,$ where $H$ is an intermediate order, then $\rho\in M(R_F)\backslash S(R_F) \implies (R_F)_\rho$ is a DVR and hence equal to $(\OK)_v$ for the unique prime ideal $v\in M(\OK)$ containing $\rho$ and thus, there is a unique prime in $M(H)$ over $\rho$, which is outside $S(H).$ 

    For $\rho \in S(R_F)$, \cref{Normefrelation} tells us 
    \[
    2=ef(\rho)=\sum_{\substack{\delta\in M(H)\\\delta\cap R_F=\rho}} ef(\delta).
    \]
    Thus, at most one $\delta\in M(H)$ over $\rho$ can have $ef(\delta)=2$, and hence $H$ is a ReSMO.  
\end{proof}

\section{Defining the polynomial, $\pounds_n$ and correspondence to count squarefree values}

\begin{notation}
   Let  $T_n:=\Z[a_0,a_1,\cdots,a_{n-2}, b_1, b_2,m]$.
\begin{definition}\label{defining pounds}
       We define $\pounds_n\in T_n$ by    
        \[
        \pounds_n(a_0,a_1,\cdots,a_{n-2},b_1,b_2,m):=\frac{\disc(a_0x^n+a_1x^{n-1}y+\cdots+a_{n-2}x^2y^{n-2}+mb_1x^{n-1}y+m^2b_ny^n)}{m^2}.
\]
        We will refer to $a_0x^n+a_1x^{n-1}y+\cdots+a_{n-2}x^2y^{n-2}+mb_1x^{n-1}y+m^2b_ny^n$ as the binary $n$-ic form associated to the point $(a_0,a_1,\cdots,a_{n-2},b_1,b_2,m)$.
\end{definition}

\end{notation}
\begin{lemma}
    It follows from Lemma 3 (Disc structure final two) in \cite{patilEkedahlSingularDiscriminant} that $\pounds_n$ is polynomial with integer co-efficient in the variables, $a_0,a_1,\cdots,a_{n-2},b_1,b_2,m.$
\end{lemma}
Note that $\pounds_n$ is a discriminant polynomial for the weakly divisible parameter space.
Similar to what the discriminant polynomial is to the parameter space of binary forms. 
\begin{theorem}\label{symmetriesofpounds}(Symmetries of $\pounds_n$).
The polynomial $\pounds_n(*,m)$ is invariant under the action of $G_m\subseteq SL_2(\Z)$ where $G_m$ is defined by elements of $g\in SL_2(\Z)$ satisfying $g\equiv id \bmod m.$
\end{theorem}
\begin{proof}
    Follows from definition of $\pounds_n$ via discriminant polynomial and how $GL_2$ action affects the discriminant.
\end{proof}
The following is another version of \cref{LRGST M}.
\begin{theorem}
    If $f$, a binary $n$-ic form, is ultra weakly divisible, then there exists unique $m_f$ and $l_f\in \P(\Z/m_f\Z)$ such that $f(L_f(x,y)) = a_0x^n+a_1x^{n-1}y+\cdots+a_{n-2}x^2y^{n-2} + m_fb_1xy^{n-1} +m_f^2b_2y^n$  such that $\pounds_n(a_0,a_1,\cdots,a_{n-2},b_1,b_2,m)$ is squarefree, where $L_f$ is any transformation in $GL_2(\Z)$ that is a lift of the transformation taking $[1:0]\in \P(\Z/m_f\Z)$ to $l_f\in \P(\Z/m_f\Z)$.
\end{theorem}
\begin{proof}
    Follows from \cref{discriminint of R'} and \cref{defining pounds} and \cref{LRGST M}.
\end{proof}
\begin{discussion}
    Thus, distinct $GL_2(\Z)$ orbits of ultra-weakly divisible polynomials correspond to distinct points (up to \cref{symmetriesofpounds}) at which $\pounds_n$ takes squarefree values. And ultra-weakly divisible polynomials are polynomials that do not, upon reduction modulo any $p$, lie in the singular locus of the discriminant. This will allow us to count squarefree values of $\pounds_n$ as Ekedahl sieve-able elements in the space of forms. This should hold for all multivariate polynomials, though the relative distributions will determine if this strategy will work for counting squarefree values of other forms that one might be able to generate from a given form.
\end{discussion}
\begin{remark}
    The converse of this theorem is almost true, more precisely, if $\pounds_n$ is squarefree then corresponding polynomial $f$ is not strongly divisible at all primes not dividing $m$.
    Some similarly restrictive is true for primes $p$ dividing $m$, though we need to take in an extremely small sliver of polynomials that are restrictively strongly divisible at primes divisible by $m$ such polynomials which also give squarefree values at $\pounds.$ These are the polynomials missed by our approach but exponent wise, these are inconsequential. They will have an effect on the actual constant of proportionality if someone cares about that. 
\end{remark}
\subsection{Local densities of Ultra Weakly Divisible Forms}
We now compute the number of binary forms $\bmod p$ such that they are not strongly divisible.
From \cref{Fundamental theorem of rings and polynomials}, we know that we wish to count binary forms with a linear double root and no other double roots.
We will then add the condition of $p\nmid (a_0,a_1).$

\begin{theorem}\label{NSD local density}
    The number of binary forms of degree $n\ge 4$ that are not strongly divisible $\bmod p$ (not in the singular locus of the discriminant or has no triple linear factor or two linear factors in $\overline{\F}_p$) is 
    \[
    p^{n+1} \left(1-\frac{1}{p^2}\right)^2.
\]
\end{theorem}
\begin{proof}
    We will divide this set into 3 subsets:
    \begin{itemize}
        \item The first set consists of $f$ such that $p\nmid \disc(f)$ and $p\nmid a_0$.
        We know from \cite{SQFREE-Prob} that the probability of $p\nmid \disc(f)$ for $f$ monic of degree $\ge 2$ is $1-p^{-1}.$  Binary forms in this set are just homogenizations of polynomials of the form $cg$ with $\disc(g)\neq 0$ where $c\in \F_p^*$.
        Thus the number of binary forms in this set is
        \[
        p^{n+1}(1-p^{-1})^2.
        \]
        \item The second set consists of $f$ such that $p\nmid \disc(f)$ and $p|a_0$.
        If $p|a_0$ and $p|a_1$, then clearly $f$ has a double root in $\P(\F_p)$.
        Thus we restrict ourselves to the case $p\nmid a_1.$ Since $\disc(f)=a_1^2\disc(f/y)$, it follows that we are again looking for elements that are $y$-times the homogenization of a polynomial of the form $cg$ where $g$ is a monic polynomial of degree $n-1$  with $\disc(g)\neq 0$ and $c\in \F_p^*.$ Thus, the number of such binary forms is 
        \[
        p^{n+1}(1-p^{-1})^2(1/p).
        \]
        \item The third set consists of $f$ such that $p|\disc(f)$.
        Suppose it has a double root at $0$; then we wish to compute binary forms $h(x)x^2$, such that $h(x)$ has no other double root or $p\nmid \disc(h)$ and $p\nmid h(0).$ Thus, these are $x^2$ times the palindromic reciprocal of a homogenization of a polynomial of the type $cg$, where $g$ is a monic polynomial of degree $n-2$ with $\disc(g)\neq 0$ and $c\in \F_p^*.$ Thus, as long as $n-2\ge 2$, the number of such polynomials is
        \[
        p^{n+1}(1-p^{-1})^2 (1/p^2).
        \]
        Since the polynomials that have a double root at $l$ and are not strongly divisible by $p$ are in bijection with this set and have no overlap (overlap would imply two double roots and thus strong divisibility), the number of elements in this set is
        \[
        p^{n+1}(1-p^{-1})^2 (1/p^2)(p+1).
        \]
    \end{itemize}
    Summing over them gives the result.
\end{proof}
\begin{theorem}\label{NSD with coprime}
    The number of binary forms of degree $n\ge 4$ that are not strongly divisible $\bmod p$ (not in the singular locus of the discriminant or has no triple linear factor or two linear factors in $\overline{\F}_p$) such that $a_0\neq0$ or $a_1\neq0$ is 
    \[
    p^{n+1} \left(1-\frac{3}{p^2}+\frac{1}{p^3}+\frac{1}{p^4}\right).
\]
\end{theorem}
\begin{proof}
    We instead compute the number of binary forms which are not strongly divisible by $p$ and $a_0=0$, $a_1=0.$ These are binary forms of the form $y^2 h$ where $h$ is a homogenization of $cg$ where $g$ is a monic polynomial with $\disc(g)\neq 0$ and $c\in \F_p^*.$ The result follows.
\end{proof}
By exactly the same argument as above, we have the following:
\begin{lemma}\label{p|m}
    The number of binary forms weakly divisible by $p$ at $l$ and not strongly divisible by $p$ is
    \[
    (1-p^{-1})^2p^{-3}
    \]
\end{lemma}\subsection{Count.}
Suppose $m$ is squarefree.
We refer to elements in $\P(\Z/m\Z)$ by $(m_1,l,m_2)$ as explained in \cref{P(Z/mZ)representation}.
For such an element, the number of binary forms $f=a_0x^n+a_1x^{n-1}y+\cdots+y^na_n$ such that $f$ is ultra weakly divisible for primes less than $M$ is bounded by 
\[
\textit{number of solutions }\bmod m^2\prod_{\substack{p\nmid m\\p<M}} p \times \textit{number of boxes of side length } m^2\prod_{\substack{p\nmid m\\p<M}} p 
\]
By the Chinese Remainder Theorem, the number of solutions can be divided into local factors.
We will further restrict ourselves to only those forms modulo 2 which are irreducible mod $2$.
Thus, we are counting the number of binary forms which are not strongly divisible by any prime $p$.
First, we count $1$-reduced forms, as they correspond to distinct binary rings and thus distinct $GL_2(\Z)$ orbits (see the theorems reviewed in the next section, \cref{reducedregion}).
Thus, the number of binary forms in a box given by $c_is\le|a_i|\le d_i s$ that are weakly divisible by $(m_1,l,m_2)$, irreducible mod 2, and not strongly divisible for all $p<M, p\nmid m$ is 
\[
\prod_{p|m}(1-p^{-1})^2p^{-3} \prod_{\substack{p<M\\p\nmid m}}(1-3p^{-2}+p^{-3}+p^{-4})\prod_{i=0}^n (d_i-c_i) s^{n+1} +O_n\left(\sum_{r=1}(P_M m^2)^r s^{n+1-r}\right).
\]
Combined with Theorem 1(Disc Ekedahl Modular) in \cite{patilEkedahlSingularDiscriminant} and Theorem 4 in \cite{patilEkedahlSingularDiscriminant}, we arrive at the following proposition for the count.

\begin{proposition}\label{prop:sieve_bound}
The number of binary forms of degree $n$ which are weakly divisible by $m$ at $(m_1,l,m_2)$ and ultra weakly divisible such that $s\ge m^2$ is bounded by $\mathcal{M} + \mathcal{T} + \mathcal{E}$, where the main term $\mathcal{M}$ is
\[
    \mathcal{M} = \prod_{p|m}(1-p^{-1})^2p^{-3} \prod_{\substack{p<M\\p\nmid m}}(1-3p^{-2}+p^{-3}+p^{-4})\prod_{i=0}^n (d_i-c_i)s^{n+1},
\]
the tail-end estimate $\mathcal{T}$ is bounded by:
\[
    \mathcal{T} \ll_n \prod_{p|m}(1-p^{-1})^2 s^{n-1} \cdot \frac{s^2}{m^3}\left(\frac{\log s}{(s/m^2)(\log(s/m^2)+\log \log s)} + \frac{1}{M\log M}\right),
\]
and the boxing error $\mathcal{E}$ is bounded by:
\[
    \mathcal{E} \ll_n \sum_{r=1}(P_M m^2)^r s^{n+1-r}.
\]
\end{proposition}

We sum over all values of $(m_1,l,m_2)$ to get the following theorem for $n\ge 5$. For $n=4$ we cannot guarantee a lack of double counting for all $(m_1,l,m_2)$; however, we can guarantee it for $(1,l,m).$
\begin{theorem}\label{nge5}
    The number of ultra-weakly divisible reduced $n$-ic binary forms with $|a_i|\le s$ for $0\le i\le n$ and $\alpha s\le a_0$ that are weakly divisible by $m$ (for squarefree $m$) is 
    \[
    \gg_n \frac{s^{n+1}}{m^2}\prod_{p|m}(1-p^{-1})
    \] 
    provided that $\frac{s}{m^2}\gg_n\frac{\log s}{\log \log s}$, $n\ge 5$ and $s\gg 1$.
\end{theorem}

\begin{theorem}\label{nge4}
    The number of ultra-weakly divisible reduced quartic forms with $|a_0|,|a_1|,|a_2|,|a_3|,|a_4|\le s$(and $a_0\ge \frac{s}{n}$ that are weakly divisible by $m$ (for squarefree $m$) is 
    \[
    \gg_n \frac{s^{5}}{m^2}(\prod_{p|m}(1-p^{-1}))^2
    \] 
    provided that $\frac{s}{m^2}\gg_n\frac{\log s}{\log \log s}$, and $s\gg 1$.
\end{theorem}
The condition $a_0\ge \alpha s$ allows us to say that there are at most $O_\alpha(1)$ elements in this set that correspond to the same translation orbit. This will allow us to move the above lower bound from forms to $GL_2(\Z)$ orbits, as every such orbit has at 
most one translation orbit of reduced form from \cref{distinct-rings-old} applied to $m=1$ and the fact that $R_f$ is invariant under $GL_2(\Z)$ on binary forms.

Since,$GL_2(\Z)$ ultra weakly divisible orbits correspond to distinct points for which $\pounds_n$ takes a squarefree value that is less than $s^{2n-2}/m^2,$ we set $X=s^{2n-2}/m^2$.
The condition $\frac{s}{m^2}\gg_n\frac{\log s}{\log \log s}$ becomes $\frac{(m^2X)^{1/(2n-2)}}{m^2}\gg \frac{\log X}{\log\log X}.$ Simplifying, we get
\[
(X^{\frac{1}{2n-2}}\frac{\log \log X}{\log X})^{\frac{n-1}{2n-3}}\gg_n m.
\]
Thus, we can say the following.
\begin{theorem}
    The number of ultra-weakly divisible $GL_2(\Z)$ orbits that intersect the space defined by $|a_i|\le (m^2X)^{1/(2n-2)}$ for all $0\le i \le n,$ that are weakly divisible by $m$ (for squarefree $m$) is 
    \[
    \gg_n \frac{(m^2X)^{\frac{1}{2}+\frac{1}{n-1}}}{m^2\log m}
    \] 
    provided that 
    \[
    X^{\frac{1}{4n-6}}(\frac{\log\log X}{\log X})^{\frac{n-1}{2n-3}}\gg_n m
    \]
    and $n\ge 5.$
\end{theorem}
For the following theorem, it is better to restrict $m$ to the set of primes (the most 
ideal set to restrict to would be squarefree numbers without too many prime factors).
Secondly, we will only sum over primes in $(CX^{\frac{1}{4n-6}}(\frac{\log\log X}{\log X})^{\frac{n-1}{2n-3}}, DX^{\frac{1}{4n-6}}(\frac{\log\log X}{\log X})^{\frac{n-1}{2n-3}})$ for some $C>D$.
\[
\#\{(f,m): f\textit{ is $1$-reduced}, f \textit{ is weakly divisible at m}, \frac{\disc(f)}{m^2} <X\}\gg \frac{X^{\frac{1}{2}+\frac{1}{n-1.5}}(\log \log X)^{\frac{1}{n-1.5}}}{(\log X)^{1+\frac{1}{n-1.5}}}.
\]

Now given a ultra-weakly divisible reduced $1$-form $f$, the number of times  $f$ is counted above is at-most $\omega(m_f).$ However we can do better as we have only sum over certain primes which are actually in $(CX^{\frac{1}{4n-6}}(\frac{\log\log X}{\log X})^{\frac{n-1}{2n-3}}, DX^{\frac{1}{4n-6}}(\frac{\log\log X}{\log X})^{\frac{n-1}{2n-3}})$.
Thus, $f$ cannot be counted more than $4n-3$ times, as this would result in the discriminant of $f$ to be too big.
Thus, the number of distinct ultra weakly divisible orbits which contribute to a squarefree value of $\pounds_n$ that is $\le X$ is
\[
\gg\frac{X^{\frac{1}{2}+\frac{1}{n-1.5}}(\log \log X)^{\frac{1}{n-1.5}}}{(\log X)^{1+\frac{1}{n-1.5}}}.
\]
Thus, we get the following theorem.
\begin{theorem}
If $S(X)$ denotes a set of points $P$ such that $\pounds_n(P)$ is squarefree and $\le X$ such that for any distinct points $Q,Q'\in S(X)$, $Q$ is not equivalent to $Q'$ up to the action given in \cref{symmetriesofpounds}, then 
\[
\#S(X)\gg_n\frac{X^{\frac{1}{2}+\frac{1}{n-1.5}}(\log \log X)^{\frac{1}{n-1.5}}}{(\log X)^{1+\frac{1}{n-1.5}}}
\]
provided that $n\ge 5.$
\end{theorem}
\begin{remark}
    One can easily optimize this by managing the number of prime factors of the squarefree numbers to ensure the same reduced translation orbit isn't counted too many times, might end up adding a log log term in denominator at the cost of removal of a log 
term in denominator, and sum over all squarefree values in the above range.
The full removal of $(\log\log/\log)^{\frac{1}{n-1.5}}$ however will require an improvement on the version of Ekedahl sieve used, which the authors suspect will be a bottleneck.
\end{remark}

\section{Global Counts and Lower Bounds}

\subsection{Reduced-m-polynomials/Reduced polynomials and uniqueness of Rings}
This section is a summary of results from our previous paper \cite{patil2024weaklydivisiblerings} which in turn is an adaptation/variation of ideas in \cite{BSW-monicsquarefree} concerning distinct translation orbits of monic polynomials corresponding to distinct rings.
\begin{definition}
    If $\ncp{v_1,v_2,\cdots,v_n}$ is a basis for $\R^n$ and $v_i'$ denotes the projection of $v_i$ to the space orthogonal to the space spanned by $\{v_1,v_2,\cdots, v_{i-1}\}$ and $t_i=||v_i'||,$ then we say   $\ncp{v_1,v_2,\cdots,v_n}$ is {\em Normally Minkowski Reduced} if the $t_i$'s satisfy 
    \begin{equation}\label{first 2 minkow}
    \frac{t_i}{t_2}\ge 2 \textit{ for all } 3\le i\le n \textit{ and  } \frac{t_2}{t_1}\ge 2
    \end{equation}
\end{definition}
\begin{definition}
    We say a (real) polynomial $f$ is a Reduced-$m$-polynomial if
    \begin{itemize}     
        \item The canonical basis for $R'_{(f,m)}$ is Normally Minkowski Reduced when seen as $R'_{(f,m)}\subseteq \R^r\oplus\C^s$ under the canonical norm.
\end{itemize}
\end{definition}
\begin{remark}
     $f$ is a Reduced-$m$-polynomial if and only if for any real $r,$ $f_r$ is also a Reduced-$m$-polynomial.
    Furthermore, the fact that the standard basis for $R_{(f,m)}$ is Normally Minkowski Reduced implies that $R_{f,m_1,m_2}$ is Normally Minkowski Reduced  for any $m_1, m_2$ satisfying $m_1m_2=m.$  
\end{remark}
\begin{theorem}\label{distinct-rings-old}
    Given polynomials $f$ and $g$ of degree $n\ge 4$ such that $f$ is a Reduced-$m$-polynomial weakly divisible by $m$ at $e$, and $g$ is a Reduced-$m'$-polynomial weakly divisible by $m'$ at $d$ ($m,m'\ge 1$) with $R'_{(f_e,m)}= R'_{(g_d,m)}$, then 
    \[
    m=m' \textit{ and }\exists r\in \Z: g(x)=f(x+mr-d+e).
    \] 
\end{theorem}
Naturally, this slight modification of the given theorem also holds.
\begin{theorem}\label{distinct-rings}
    Given polynomials $f$ and $g$ of degree $n\ge 5$ such that $f$ is a Reduced-$m$-polynomial weakly divisible by $m$ at $(m_1,e,m_2)$, and $g$ is a Reduced-$m'$-polynomial weakly divisible by $m'$ at $(m_1',d,m_2')$ ($m,m'\ge 1$) with $R'_{(f_e,m_1,m_2)}= R'_{(g_d,m_1',m_2')}$, then 
    \[
    m_i=m_i' \textit{ and }\exists r\in \Z: g(x)=f(x+m_2r-d+e).
    \] 
    In particular, $m=m'$ and $(m_1,e,m_2)=(m_1',d,m_2')\in \P(\Z/m\Z).$
\end{theorem}
Note the necessity of $n\ge 5$ coming from the consideration of $\ncp{1,B_1,B_2,B_3}$ to establish $m_1'=m_1$ and $a_0=b_0.$

\begin{lemma}\label{minkow}If $f$ is a real monic polynomial with non-zero discriminant, then there exists a $\rho_f\in \R^{+ve}$ (continuously varying with $f$) such that $\lambda\rho^{n} f(\frac{x}{\rho})$ is a Reduced-$1$-polynomial for all $\rho\ge \rho_f$ and $\lambda\ge 1$.
If $\rho \ge m^{1/(n-2)}\rho_f$ where $m>0$ and $\lambda\ge 1$, the polynomial $\lambda\rho^{n} f(\frac{x}{\rho})$ is a Reduced-$m$-polynomial.
\end{lemma}

\begin{definition}
    We define the height of a binary $n$-ic form $f=a_0x^n+a_1x^{n-1}y+\cdots a_ny^n$, satisfying $a_0\neq 0$, to be
    \[
    H(f)=\max\{|a_i/a_0|^{1/i}\}.
    \]
\end{definition}
The following theorems follow from standard arguments from the previously mentioned papers. One starts with a polynomial (monic) $f$ of degree $n$ and height $1$ with a non-zero discriminant, say $x^n+1$. Since the discriminant is a polynomial, one can find a neighbourhood box (also of height one, consisting of monic polynomials) of $x^n+1$ such that the discriminant remains positive in the box, say $B$, given by $|a_i|\le c_i <1$ for all $1\le i\le n-1.$ Since $\rho_f$ is a continuous function on this compact region, it has a maximum in this region, say $\rho_B.$ Then, forms that are deformed (via $f\longrightarrow f(H(f)x)/H(f)^n$) into forms of height one in box $B$ with height greater than or equal to $\rho_B$ are reduced monic polynomials. And thus, homogenizations of polynomials of the type $\lambda f$, where $\lambda\ge 1$ and $f$ is a reduced polynomial, are reduced. Similarly, forms that are deformed (via $f\longrightarrow f(H(f)x)/H(f)^n$) into forms of height one in box $B$ with height greater than or equal to $m^{1/(n-2)}\rho_B$ are reduced monic polynomials. And thus, homogenizations of polynomial of the type $\lambda f$, where $\lambda\ge 1$ and $f$ is a reduced polynomial, are reduced forms. We thus get the following theorems.
\begin{theorem}\label{reducedregion}
    There exist constants $c_i$ and $d_i$ such that every form $a_0x^n+a_1x^{n-1}y+\cdots +a_ny^n$ is $1$-reduced if for all $0\le i \le n$, we have $c_is <a_i<d_is$. We will assume $c_0>0.$
\end{theorem}
\begin{theorem}\label{reducedmmonic}
    There exist positive constants $\alpha_i$ such that every monic polynomial $x^n+a_1x^{n-1}+\cdots+a_{n-1}x+a_n$ is $m$-reduced 
if for all $1\le i\le n-1$, we have $|a_i|\le \alpha_i t^i, \alpha_nt^n\le  a_n\le\alpha_{n+1}t^{n}$ and $cm^{n/(n-2)}\le \alpha_nt^n$.
\end{theorem}
\begin{theorem}\label{reducedmforms}
    There exist positive constants $\mathcal{C}_i,\mathcal{D}_i$ such that every monic polynomial $a_0x^n+a_1x^{n-1} y + \cdots + a_{n-1} xy^{n-1} + a_ny^n$ is $m$-reduced if for all $0\le i\le n-1,$ we have $\mathcal{C}_i st^i\le a_i\le \mathcal{D}_i st^i$ and  $\mathcal{C}_nst^n\le  a_n\le \mathcal{D}_nst^{n}$ and $ m^{n/(n-2)}\le \mathcal{C}_nt^n.$ We will assume $\mathcal{C}_0>0.$
\end{theorem}

\subsection{Counting distinct weakly divisible rings coming from ultra weakly divisible polynomials}

In this section, we wish to compute $m$-reduced ultra weakly divisible forms.
We consider two cases: one where the form is monic, which will correspond to number-fields of the type $\Z[\alpha,\beta]$ where $\alpha\cdot\beta \in \Z.$ This should be of independent interest due to the structure and will be the best $\limsup$ bound for number-fields with bounded discriminant until the next section, where we will count number-fields whose ring of integers is of the type $\Proj(\Z[X,Y,Z]/I),$ where $I$ is a homogeneous ideal containing an element of the form $XY-aZ^2$ for some $a\in \Z.$ We will also reinterpret these in terms of counting unramified $A_n$ extensions over $\Q(\sqrt{s})$ for $s\le X.$

We count in the box defined by \cref{reducedmforms}, using the exact same procedure as above, but with a minor change.
Since we wish to make the theorem true for as large an $m$ as possible, we will actually only consider boxes of side length $P_M:=\prod_{p<M}p$, and we will finally give an upper bound for the number of forms that are weakly divisible by $m$ but not strongly divisible by $m$ by fixing the first $n-3$ coefficients and bounding the number of potential values for the $(n-2)^{th}$ coefficient.
In other words, we make boxes of dimensions $
(P_M,P_M,\cdots,P_M, mP_M, mP_M, m^2P_M)$.

Given a solution $\bmod P_M$ and its unique lifts in the chosen box for the first $n-3$ coefficients given by $R_1,R_2,\cdots,R_{n-3}$, we consider the set of elements in $\Z/m\Z$ which do not solve 
$T\disc(X^{n-2}+R_1X^{n-3}+\cdots+R_{n-3}X+T)\equiv 0\bmod m$.
Clearly, this will have at most $n-1$ solutions modulo a prime (the polynomial in $T$ is always degree $n-2$ as long as $\gcd(n,m)=1$, from \cref{Structuredisc} applied to monic polynomials). For any such non-solution, in every box above we have a unique lift, as there will be a unique element modulo $mP_M$ satisfying this property.

Similarly, once we have chosen $R_1,R_2,\cdots, R_{n-2}$, we do the exact same thing to get the right choice for $R_{n-1}$ (where $R_{n-1}$ is chosen to be a lift of $-nl^{n-1}-(n-1)R_1l^{n-2}-\cdots -2R_{n-2}l \bmod m$ and the solution chosen modulo $P_M$), and $R_n$ (where $R_n$ is chosen to be a lift of $-l^{n}-R_1l^{n-2}-\cdots -R_{n-2}l^2-R_{n-1}l \bmod m^2$ and our chosen solution modulo $P_M$).
Basically, given a solution $\bmod P_M$ and its unique lifts in the chosen box for the first $n-3$ coefficients given by $R_1,R_2,\cdots,R_{n-3}$, we consider the set of elements in $\Z/m\Z$ which do not solve 
$T\disc(X^{n-2}+R_1X^{n-3}+\cdots+R_{n-3}X+T)\equiv 0\bmod m$.
Clearly, this will have at most $n-1$ solutions modulo a prime (the polynomial in $T$ is always degree $n-2$ as long as $\gcd(n,m)=1$, from \cref{Structuredisc} applied to monic polynomials). For any such non-solution, in every box above we have a unique lift, as there will be a unique element modulo $mP_M$ satisfying this property.

Note that the probability of a monic polynomial being not strongly divisible $\bmod p$ is 
\[
(1-1/p^2).
\]
We thus get the following.

\begin{theorem}
    The number of ultra weakly divisible reduced-$m$-monic polynomials with height $\le t$ that are weakly divisible by $m$ (assuming $p|m\implies p>n$) is 
    \[
    \gg_n \frac{t^{\frac{n(n+1)}{2}}}{m^2}\prod_{p|m}\left(1-\frac{n-1}{p}\right)
    \]
    provided that $\frac{t^n}{m^2}\gg \frac{\log t}{\log \log t}.$
\end{theorem}

Since each translation orbit can have at most $O(t)$ reduced-$m$-monic polynomials among monic polynomials with height $\le t$, we can say the following.
\begin{theorem}
    The number of distinct ultra weakly divisible reduced-$m$-monic polynomial translation orbits that are weakly divisible by $m$ and intersect with the space of monic polynomials with height $\le t$ is 
    \[
    \gg_n \frac{t^{\frac{n(n+1)}{2}-1}}{m^2}\prod_{p|m}\left(1-\frac{n-1}{p}\right)
    \]
    provided that $\frac{t^n}{m^2}\gg \frac{\log t}{\log \log t}.$
\end{theorem}
Since all of these correspond to an arithmetic object with discriminant $\le t^{n(n-1)}/m^2$, we set $m^2X=t^{n(n-1)}$ to get the following theorem.
\begin{theorem}
    The number of distinct ultra weakly divisible reduced-$m$-monic polynomial translation orbits that are weakly divisible by $m$ and correspond to a weakly divisible ring of discriminant $\le X$ is 
    \[
    \gg_n \frac{X^{\frac{1}{2}+\frac{1}{n}}}{m^{1-\frac{2}{n}}}\prod_{p|m}\left(1-\frac{n-1}{p}\right)
    \]
    provided that $\frac{X^{\frac{1}{n-1}}}{m^{2-\frac{2}{n-1}}}\gg \frac{\log X}{\log \log X}.$
\end{theorem}
Summing over $m$ in the given range gives the following.

\begin{theorem}
    The number of distinct ultra weakly divisible rings with squarefree discriminant $\le X$ associated to monic polynomials is 
    \[
    \gg_n X^{\frac{1}{2}+\frac{1}{n-1-\frac{1}{n-1}}}\left(\frac{\log\log X}{\log X}\right)^{\frac{1}{n-1-\frac{1}{n-1}}}.
    \]
\end{theorem}
\begin{discussion}
    Since these are all ReSMOs, \cref{monicUWD} combined with \cref{LRGST M} tells us these are all ReSMOs in a number-field $\K$ with squarefree discriminant such that $\OK=\Z[\alpha,\beta]$ where $\alpha\beta\in \Z.$ Summing \cref{Number of ReSMOs} over fields of squarefree discriminant bounded by $X$ to dominate this value gives us the following theorem.
\end{discussion}

\begin{theorem}
    If $a_s$ denotes the number of number-fields with squarefree discriminant $s$ whose ring of integers can be written as $\Z[\alpha,\beta]$ where $\alpha\beta\in \Z,$ then 
    \[
    \sum_{s\le X} \frac{a_s}{\sqrt{s}}\gg_n X^{\frac{1}{n-1}+\frac{1}{n(n-1)(n-2)}}\frac{(\log\log X)^{\frac{1}{n-1}+\frac{1}{n(n-1)(n-2)}}}{(\log X)^{\binom{n}{2}-1+\frac{1}{n-1}+\frac{1}{n(n-1)(n-2)}}}.
    \]
\end{theorem}
\begin{discussion}\label{UnweightedUnramifiedA_n}
    Note that applying the tail end estimates in \cite{BSW-monicsquarefree} to count weakly divisible rings associated to monic polynomials as done to binary forms in \cite{patil2024weaklydivisiblerings} (when seen in the context of \cref{monicUWD}) gives us that the number of number-fields $\K$ such that $\OK=\Z[\alpha,\beta]$ (we can add the additional condition that $\disc(\OK)$ is squarefree for an unramified $A_n$ extension of $\Q(\sqrt{s})$ interpretation below) where $\alpha\beta\in \Z$ and $\disc(\OK)\le X$ is 
    \[
    \gg_n X^{\frac{1}{2}+\frac{1}{n-\frac{4}{5(n-1)}}}
    \]
    when $n\ge 3$.
    On the other hand, the above theorem gives the best exponent $c_n$ for when 
    \[
    \limsup\frac{N(n,X)}{X^{c_n-\epsilon}}=\infty.
    \]
    Linnik's conjecture tells us $N(n,X)\sim r_nX,$ and thus the largest $c_n$ expected here is $1$. We will improve this in the next main result. 
\end{discussion}

We now apply this setup to count weakly divisible rings associated to ultra-weakly divisible binary forms. To keep the rings distinct, we will count $m$-reduced forms that are ultra-weakly divisible and weakly divisible by $m.$ More specifically, we apply it to the region given by \cref{reducedmforms}. However, we will just set $t=m^{1/(n-2)}/(\mathcal{C}_n)^{1/n}$ in \cref{reducedmforms} and apply the exact same procedure as above, making boxes of side lengths $(P_M,P_M,\cdots,P_M, mP_M, mP_M, m^2P_M)$ and applying Theorem 1(Disc Ekedahl Modular) in \cite{patilEkedahlSingularDiscriminant} with $M=M$, $A=\mathcal{D}_{n-1}sm^{(n-1)/(n-2)}/(\mathcal{C}_n)^{(n-1)/n}$ and $B=\mathcal{D}_{n}sm^{n/(n-2)}/\mathcal{C}_n$ to obtain the following.

\begin{theorem}
    The number of distinct ultra-weakly divisible reduced-$m$-forms that are weakly divisible by $m$ and lie in the space given by \cref{reducedmforms}, where $t\sim m^{1/(n-2)}$, is 
    \[
    \gg_n \frac{s^{n+1}m^{\frac{n(n+1)}{2(n-2)}}}{m^2}\prod_{p|m}\left(1-\frac{n-1}{p}\right)
    \]
    provided that $\frac{sm^{n/(n-2)}}{m^2}\gg \frac{\log s}{\log \log s}.$
\end{theorem}
Since each translation orbit can have at most $O(m^{1/(n-2)})$ reduced-$m$-forms in the given space, we get the following.
\begin{theorem}
    The number of distinct ultra-weakly divisible reduced-$m$-form translation orbits that are weakly divisible by $m$ and intersect with the space given by \cref{reducedmforms} is 
    \[
    \gg_n \frac{s^{n+1}m^{\frac{n(n+1)-2}{2(n-2)}}}{m^2}\prod_{p|m}\left(1-\frac{n-1}{p}\right)
    \]
    provided that $\frac{sm^{n/(n-2)}}{m^2}\gg_n \frac{\log s}{\log \log s}.$
\end{theorem}

Since each such object corresponds to a distinct ring with discriminant $\ll_n s^{2n-2}m^{n(n-1)/(n-2)}/m^2$, we set $m^2X\sim s^{2n-2}m^{n(n-1)/(n-2)}$ to count rings weakly divisible by $m$ that are associated to $m$-reduced ultra-weakly divisible forms with discriminant less than $X.$
\begin{theorem}
    The number of rings that are weakly divisible by $m$ with discriminant less than $X$, associated to an ultra-weakly divisible $m$-reduced form is 
    \[
    \gg_n X^{\frac{1}{2}+\frac{1}{n-1}}m^{\frac{2}{(n-1)(n-2)}-1}\prod_{p|m}\left(1-\frac{n-1}{p}\right)
    \]
    provided that $X^{\frac{n-2}{3n^2-13n+12}}\left(\frac{\log X}{\log \log X}\right)^{\frac{2(n-1)(n-2)}{3n^2-13n+12}}\gg_n m.$
\end{theorem}
Since distinct $m$ values result in distinct rings, from \cref{distinct-rings}, we can sum over all $m$ to count distinct weakly divisible rings associated to ultra-weakly divisible forms.
\begin{theorem}
    The number of distinct weakly divisible rings with discriminant less than $X$, associated to some ultra-weakly divisible forms of degree $n,$ for $n\ge 5$ is 
    \[
    \gg_n X^{\frac{1}{2}+\frac{1}{n-\frac{4}{3}}} \left(\frac{\log\log X}{\log X}\right)^{\frac{2n^2-6n+4}{3n^2-13n+12}}
    \]
\end{theorem}
Since every such ring is a ReSMO (see \cref{UWDisReSMO}), we can again use \cref{Number of ReSMOs} to get a weighted count of the number of degree-$n$ number-fields of squarefree discriminant.
Note that these are number-fields whose rings of integers look like $\Proj(\Z[X,Y,Z]/I)$ where $I$ contains a form of the type $bXY-aZ^2$.
It is more restrictive than that, actually.
\begin{theorem}\label{final}
   If $b_s$ denotes the number of number-fields of degree $n$ with discriminant $s$ such that its ring of integers is $\Proj(\Z[X,Y,Z]/I)$ where $I$ contains a quadratic of the form $aXY-bZ^2$, for a squarefree number $s$, then
    \[
    \sum_{s\le X}\frac{b_s}{\sqrt{s}}\gg_n X^{\frac{1}{2}+\frac{1}{n-\frac{4}{3}}} \frac{(\log\log X)^{\frac{2n^2-6n+4}{3n^2-13n+12}}}{(\log X)^{\binom{n}{2}-1 +\frac{2n^2-6n+4}{3n^2-13n+12}}}.
    \]
\end{theorem}
\begin{corollary}
    \[
    \limsup_{X\longrightarrow\infty} \frac{N(n,X)}{X^{\frac{1}{2}+\frac{1}{n-4/3}-\epsilon}}=\infty.
    \]
\end{corollary}
\section{Unramified $A_n$-extensions of $\Q(\sqrt{s})$ as $s$ varies}
This section is motivated by \cite{bhargava2022squarefree}, and is natural as we always count number-fields with squarefree discriminant $s$, which can be interpreted as an unramified $A_n$ extension of $\Q(\sqrt{s})$. To formalize this:

    \begin{definition}
        We define $Cl_\K[G]$ to denote the Galois group of a field extension of $\K$ that is the concatenation of all unramified $H$ extensions for all $H$ for which there exists a surjective group homomorphism to $H$ from $G.$ 
    \end{definition}
Thus, $Cl_\K[C_n]=Cl_\K[n]$ and for a simple group $G$, $Cl_\K[G]$ will be $G^k$, where $k$ is the number of (distinct) unramified $G$-extensions of $\K.$ We naturally conjecture the following, which is true for abelian groups for all number-fields and true for all groups over $\Q$.
\begin{conjecture}
    For any number-field $\K,$ $Cl_\K[G]$ is finite. 
\end{conjecture}
Thus, degree-$n$ $S_n$-extensions of $\Q$ with squarefree discriminant, say $s$, are in one-to-one correspondence with degree-$n$ $A_n$-unramified extensions of $\Q(\sqrt{s})$. 

    \begin{corollary}
        $\sum_{s<X}\log|Cl_{\Q(\sqrt{s})}[A_n]|\gg_n X^{\frac{1}{2}+\frac{1}{n-\frac{4}{5(n-1)}}}.$
    \end{corollary}

    Note that Theorem 4 in \cite{bhargava2022squarefree} can be equivalently quoted as follows,
    \begin{theorem}
        $\sum_{s<X}\log|Cl_{\Q(\sqrt{s})}[A_n]|\gg_n X^{\frac{1}{2}+\frac{1}{n-1}}.$
    \end{theorem}
These results, as well as previous ones, originate from very specific geometries.
Before \cite{bhargava2022squarefree}, all results were based on monogenic number fields which are rings naturally associated to points in $\A^1(\Zbar).$ In \cite{bhargava2022squarefree}, they use rings naturally associated to $\P^1(\Zbar)=\P^1(\Qbar).$ Our previous work \cite{patil2024weaklydivisiblerings} uses number fields whose rings of integers are naturally associated to points on $\P^2(\Zbar)=\P^2(\Qbar)$ satisfying $aXY-bZ^2$.

Thus, as a corollary of our final result (\cref{final}), we get:
\begin{corollary}
    \[
    \sum_{s<X}\frac{\log|Cl_{\Q(\sqrt{s})}[A_n]|}{\sqrt
    {s}}\gg_n X^{\frac{1}{n-\frac{4}{3}}}\frac{(\log\log X)^{\frac{2n^2-6n
    +4}{3n^2-13n+12}}}{(\log X)^{\binom{n}{2}-1 +\frac{2n^2-6n
    +4}{3n^2-13n+12}}}
    \]
\end{corollary}
\section{Acknowledgements}
We would like to thank Professor Kumar Murty, Professor Arul Shankar, Professor Jacob Tsimmerman for discussions and insights. We would also like to thank Dr. Andy O'Desky for help with a mistaken assertion in a previous draft.

\bibliographystyle{abbrv}
\bibliography{main}
\end{document}